\documentclass[11pt]{amsart}

\usepackage[T1]{fontenc}
\usepackage{mathpazo}
\usepackage{amssymb}
\usepackage{microtype}
\usepackage{mathtools}
\usepackage[hidelinks]{hyperref}

\newcommand{\MCG}{\operatorname{MCG}}
\newcommand{\Sp}{\operatorname{Sp}}
\newcommand{\Stab}{\operatorname{Stab}}
\newcommand{\im}{\operatorname{im}}
\newcommand{\irank}{\operatorname{irank}}
\newcommand{\F}{\mathbb F}
\newcommand{\Z}{\mathbb Z}

\newtheorem{theorem}{Theorem}[section]
\newtheorem{lemma}[theorem]{Lemma}
\newtheorem{corollary}[theorem]{Corollary}
\theoremstyle{remark}
\newtheorem{remark}[theorem]{Remark}

\title{The Involution Rank of $\MCG(\Sigma_3)$ Is Four}
\author{Qilong Guo}
\address{College of Science, China University of Petroleum (Beijing)\\
No. 18 Fuxue Road, Changping District, Beijing 102249, China}
\email{guoqilong1984@hotmail.com}

\author{Dongqi Sun\textsuperscript{*}}
\address{College of Mathematical Sciences, Harbin Engineering University, Harbin 150001, P. R. China}
\email{sundq1029@hrbeu.edu.cn}
\thanks{\textsuperscript{*}Corresponding author.}


\begin{document}

\begin{abstract}
For the closed orientable surface $\Sigma_3$, we prove
\[
\irank\MCG(\Sigma_3)=4.
\]
Thus $\MCG(\Sigma_3)$ can be generated by four involutions, but not by three. This settles the genus-three case of Farb's problem on generation by involutions.
\end{abstract}

\maketitle

\section{Introduction}

An involution is an element of order exactly two. If a group $G$ can be generated by involutions, define its \emph{involution rank} by
\[
\irank(G)=\min\bigl\{n:G=\langle s_1,\ldots,s_n\rangle,\ s_i^2=1,\ s_i\ne1\bigr\}.
\]
Let $\Sigma_g$ be a closed connected oriented surface of genus $g$, and let $\MCG(\Sigma_g)$ denote its orientation-preserving mapping class group. Farb asked for sharp bounds on the number of involutions needed to generate $\MCG(\Sigma_g)$ \cite[Problem~5.8]{Farb}. Korkmaz proved that four involutions suffice for every $g\ge3$ and that three suffice for $g\ge8$ \cite{Korkmaz}. Y{\i}ld{\i}z later reduced the latter range to $g\ge6$ \cite{Yildiz}. A recent survey records $g=3,4,5$ as the remaining low-genus cases \cite[Question~4.5]{Survey}. The results of Guo \cite{GuoCompanion} for $g=4,5$ and Y{\i}ld{\i}z \cite{Yildiz} for $g\ge6$ together show that $\MCG(\Sigma_g)$ is generated by three involutions for every $g>3$. In this paper, we prove that $\MCG(\Sigma_3)$ cannot be generated by three involutions.

\begin{theorem}\label{thm:main}
The involution rank of the genus-three mapping class group is four:
\[
\irank\MCG(\Sigma_3)=4.
\]
\end{theorem}

\begin{remark}\label{rem:exceptional}
The genus-three case is different in nature from the higher-genus cases. The results for $g>3$ are obtained by constructing involutory generating sets, whereas the proof here is an obstruction coming from a finite symplectic quotient.
\end{remark}

\begin{corollary}\label{cor:Farb}
Together with the results of Guo \cite{GuoCompanion} and Y{\i}ld{\i}z \cite{Yildiz}, Theorem~\ref{thm:main} settles Farb's problem for closed orientable surfaces of genus at least three. More precisely,
\[
\irank\MCG(\Sigma_g)=
\begin{cases}
4,&g=3,\\
3,&g\ge4.
\end{cases}
\]
\end{corollary}

The proof of Theorem~\ref{thm:main} is by obstruction rather than construction. We use the mod-$2$ symplectic quotient $\Phi_2:\MCG(\Sigma_3)\twoheadrightarrow\Sp(6,\F_2)$. Braden's integral normal form shows that the image of an involution is either trivial or lies in the ATLAS class $2B$. We then use a $15$-dimensional irreducible complex representation of $\Sp(6,\F_2)$ and Scott's fixed-space inequality to show that three elements of $\{1\}\cup2B$ cannot generate $\Sp(6,\F_2)$. This gives the lower bound in Theorem~\ref{thm:main}; Korkmaz's theorem gives the upper bound. Section~2 records the symplectic and finite-group facts used in the proof, and Section~3 proves the theorem.

The same argument may be useful in other minimal torsion generation problems. The general idea is to pass to a finite quotient in which the relevant torsion elements lie in a small number of conjugacy classes, and then use fixed-space data in finite-dimensional representations to rule out small generating sets.

\section{Symplectic and finite-group input}

We use three ingredients. First, we recall the case of Braden's normal form needed below.

\begin{lemma}[Braden]\label{lem:Braden}
Let $S\in\Sp(6,\Z)$ satisfy $S^2=I_6$. Then $S$ is conjugate in $\Sp(6,\Z)$ to
\[
S_0=\begin{pmatrix}a&0\\0&a\end{pmatrix},
\qquad
a=I_p\oplus(-I_m)\oplus Q^{\oplus t},
\qquad
Q=\begin{pmatrix}0&1\\1&0\end{pmatrix},
\]
where $p,m,t\ge0$ and $p+m+2t=3$.
\end{lemma}

This is the specialization of \cite[Theorem~1.1]{Braden} to $\Sp(6,\Z)$. Since $p+m+2t=3$, one has $t\in\{0,1\}$. After reduction modulo $2$, the case $t=0$ gives the identity, while $t=1$ gives one nontrivial involution type.

We next record the finite-group data used to identify this type. Let $G=\Sp(6,\F_2)$, the symplectic group of a six-dimensional vector space over $\F_2$. It has order $1451520$ and four conjugacy classes of involutions, denoted $2A,2B,2C,2D$, with centralizer orders
\[
\begin{array}{c|cccc}
\text{class}&2A&2B&2C&2D\\ \hline
|C_G(x)|&23040&4608&1536&384.
\end{array}
\]
See the ATLAS tables \cite{Atlas}. In particular, an involution whose centralizer has order $4608$ belongs to the class $2B$.

The same character table gives the representation used below.

\begin{lemma}\label{lem:2B-fixed-space}
There is a $15$-dimensional irreducible complex $G$-module $U$ such that
\[
\operatorname{codim}_U U^x=4
\]
for every $x\in2B$.
\end{lemma}

\begin{proof}
The ATLAS character table contains an irreducible character $\chi$ with $\chi(1)=15$ and $\chi(2B)=7$. If $U$ affords $\chi$ and $x\in2B$, then $x^2=1$, so its eigenvalues on $U$ are $\pm1$. Hence
\[
\dim U^x=\frac{15+\chi(x)}2=11,
\]
so $\operatorname{codim}_U U^x=4$.
\end{proof}

The third ingredient is Scott's fixed-space inequality.

\begin{lemma}[Scott]\label{lem:Scott}
Let $U$ be a finite-dimensional complex $G$-module, and suppose that $G=\langle g_1,\ldots,g_r\rangle$ with $g_1\cdots g_r=1$. Then
\begin{equation}\label{eq:Scott-general}
\sum_{i=1}^r \operatorname{codim}_U U^{g_i}
\ge
\operatorname{codim}_U U^G+
\operatorname{codim}_{U^*}(U^*)^G.
\end{equation}
In particular, if $U$ is nontrivial and irreducible, then
\[
\sum_{i=1}^r\operatorname{codim}_U U^{g_i}\ge2\dim U.
\]
\end{lemma}

\begin{proof}
The first inequality is Scott's theorem \cite[Theorem~1]{Scott}. If $U$ is nontrivial and irreducible, then $U^G=0$; the same holds for the dual module $U^*$, so the right-hand side of \eqref{eq:Scott-general} equals $2\dim U$.
\end{proof}

These are the only symplectic and finite-group facts used in the proof.

\section{The genus-three obstruction}

Consider the homomorphisms
\begin{equation}\label{eq:Phi2}
\MCG(\Sigma_3)\xrightarrow{\ \Phi\ }\Sp(6,\Z)\xrightarrow{\ \pi_2\ }\Sp(6,\F_2),
\qquad \Phi_2:=\pi_2\circ\Phi.
\end{equation}
Here $\Phi$ is the standard symplectic representation induced by the action on $H_1(\Sigma_3;\Z)$ after choosing a symplectic basis; see \cite[Chapter~6]{FarbMargalit}. Both maps are surjective, and hence so is $\Phi_2$. We first determine the possible images of involutions under $\Phi_2$.

\begin{lemma}\label{lem:image-involution}
Let $f\in\MCG(\Sigma_3)$ be an involution. Then either $\Phi_2(f)=I_6$, or $\Phi_2(f)$ is conjugate in $\Sp(6,\F_2)$ to
\begin{equation}\label{eq:X}
X=\begin{pmatrix}P&0\\0&P\end{pmatrix},
\qquad
P=\begin{pmatrix}
1&0&0\\
0&0&1\\
0&1&0
\end{pmatrix}.
\end{equation}
\end{lemma}

\begin{proof}
Set $S=\Phi(f)\in\Sp(6,\Z)$. Since $f^2=1$, we have $S^2=I_6$, so Lemma~\ref{lem:Braden} applies. Because $p+m+2t=3$, one has $t=0$ or $1$. If $t=0$, the reduction of $a$ modulo $2$ is $I_3$, and hence $\Phi_2(f)=I_6$. If $t=1$, then $a=\varepsilon\oplus Q$ for $\varepsilon\in\{1,-1\}$, whose reduction modulo $2$ is $P$. Thus $\Phi_2(f)$ is conjugate to the matrix $X$ in \eqref{eq:X}.
\end{proof}

We next identify the conjugacy class of $X$. Let $V=\F_2^6$ with symplectic basis $e_1,e_2,e_3,f_1,f_2,f_3$, where $\langle e_i,e_j\rangle=\langle f_i,f_j\rangle=0$ and $\langle e_i,f_j\rangle=\delta_{ij}$. With respect to this basis, $X$ fixes $e_1,f_1$ and interchanges $e_2,e_3$ and $f_2,f_3$.

\begin{lemma}\label{lem:centralizer-stabilizer}
Put $N=X-I_V$ and $W=\im N$. Then $W=\operatorname{span}_{\F_2}\{e_2+e_3,f_2+f_3\}$ is a two-dimensional totally isotropic subspace of $V$, and
\[
C_{\Sp(6,\F_2)}(X)=\Stab_{\Sp(6,\F_2)}(W).
\]
\end{lemma}

\begin{proof}
Set $w_1=e_2+e_3$ and $w_2=f_2+f_3$. The action of $X$ gives $N(e_1)=N(f_1)=0$, $N(e_2)=N(e_3)=w_1$, and $N(f_2)=N(f_3)=w_2$. Thus $W=\operatorname{span}_{\F_2}\{w_1,w_2\}$, and the symplectic products among $w_1,w_2$ all vanish, so $W$ is totally isotropic.

If $g\in C_{\Sp(6,\F_2)}(X)$, then $gN=Ng$, hence $g(W)=W$. Thus every element of the centralizer stabilizes $W$, so $C_{\Sp(6,\F_2)}(X)\subseteq\Stab_{\Sp(6,\F_2)}(W)$.

Conversely, for $v=\sum_{i=1}^3\alpha_i e_i+\sum_{i=1}^3\beta_i f_i$ one has
\begin{equation}\label{eq:Nformula}
N(v)=\langle v,w_2\rangle w_1+\langle v,w_1\rangle w_2.
\end{equation}
Let $g\in\Stab_{\Sp(6,\F_2)}(W)$, and write $g(w_1)=aw_1+cw_2$ and $g(w_2)=bw_1+dw_2$. Since $g|_W$ is invertible, $ad+bc=1$. As $g$ preserves the symplectic form, \eqref{eq:Nformula} gives, for every $v\in V$,
\[
\begin{aligned}
gNg^{-1}(v)
&=\langle v,g(w_2)\rangle g(w_1)
 +\langle v,g(w_1)\rangle g(w_2)\\
&=(ad+bc)\bigl(\langle v,w_2\rangle w_1
 +\langle v,w_1\rangle w_2\bigr)\\
&=N(v).
\end{aligned}
\]
Thus $gNg^{-1}=N$, and hence $gXg^{-1}=X$.
\end{proof}

\begin{lemma}\label{lem:2B}
If $f\in\MCG(\Sigma_3)$ is an involution, then $\Phi_2(f)$ is either trivial or belongs to the class $2B$ of $\Sp(6,\F_2)$.
\end{lemma}

\begin{proof}
By Lemma~\ref{lem:image-involution}, it remains to identify the class of $X$. Let $\mathcal I$ be the set of two-dimensional totally isotropic subspaces of $V$. By Witt's lemma, $\Sp(6,\F_2)$ acts transitively on $\mathcal I$. Each member of $\mathcal I$ has $|\operatorname{GL}(2,2)|=6$ ordered bases. There are $63$ choices for the first basis vector and, once it is fixed, $2^5-2=30$ choices for the second in $u^\perp\setminus\operatorname{span}_{\F_2}\{u\}$. Thus $|\mathcal I|=63\cdot30/6=315$. Since $|\Sp(6,\F_2)|=1451520$, orbit--stabilizer gives $|\Stab_{\Sp(6,\F_2)}(W)|=4608$. By Lemma~\ref{lem:centralizer-stabilizer}, $|C_{\Sp(6,\F_2)}(X)|=4608$, and the ATLAS data in Section~2 identify $X$ as an element of $2B$.
\end{proof}

\begin{lemma}\label{lem:Sp62-obstruction}
The group $\Sp(6,\F_2)$ cannot be generated by three elements of $\{1\}\cup2B$.
\end{lemma}

\begin{proof}
Let $U$ be the $15$-dimensional irreducible complex representation recorded in Section~2. Suppose that $\Sp(6,\F_2)=\langle x_1,x_2,x_3\rangle$ with $x_i\in\{1\}\cup2B$. Then $\operatorname{codim}_U U^{x_i}\le4$ for $1\le i\le3$. Set $x_4=(x_1x_2x_3)^{-1}$. Since $x_1x_2x_3x_4=1$ and these four elements generate $\Sp(6,\F_2)$, Scott's inequality gives
\[
30=2\dim U\le\sum_{i=1}^4\operatorname{codim}_U U^{x_i}.
\]
On the other hand, $\operatorname{codim}_U U^{x_4}\le15$, so the right-hand side is at most $4+4+4+15=27$, a contradiction.
\end{proof}

\begin{proof}[Proof of Theorem~\ref{thm:main}]
Suppose, to the contrary, that $\MCG(\Sigma_3)=\langle f_1,f_2,f_3\rangle$ with $f_i^2=1$. Since $\Phi_2$ is surjective, the three elements $\Phi_2(f_1),\Phi_2(f_2),\Phi_2(f_3)$ generate $\Sp(6,\F_2)$. By Lemma~\ref{lem:2B}, each lies in $\{1\}\cup2B$, contrary to Lemma~\ref{lem:Sp62-obstruction}. Thus $\irank\MCG(\Sigma_3)\ge4$. Korkmaz proved that four involutions generate $\MCG(\Sigma_g)$ for every $g\ge3$ \cite[Theorem~1]{Korkmaz}, so
\[
\irank\MCG(\Sigma_3)=4.
\]
\end{proof}

\section*{Declaration on the use of AI}
AI-assisted tools were used only for language polishing. The author is fully responsible for the content of the manuscript.

\end{document}